\RequirePackage[l2tabu,orthodox]{nag}

\documentclass[11pt,a4paper]{amsart}

\usepackage[T1]{fontenc}
\usepackage{lmodern}
\usepackage[utf8]{inputenc}
\usepackage[english]{babel}
\usepackage{amsmath,amsfonts,amssymb,mathrsfs}
\usepackage{graphicx}
\usepackage{subcaption}
\usepackage{tikz}
\usepackage{xcolor}
\usepackage[expansion=false]{microtype}
\usepackage{hyperref}

\tikzstyle{vertex}=[circle,fill=white,draw,inner sep=0pt,minimum size=5pt]
\tikzstyle{vortex}=[circle,fill=lightgray,draw,inner sep=0pt,minimum size=5pt]
\tikzstyle{vartex}=[circle,fill=black,draw,inner sep=0pt,minimum size=5pt]

\tikzstyle{bvertex}=[circle,fill=white,draw,inner sep=0pt,minimum size=3pt]
\tikzstyle{bvortex}=[circle,fill=lightgray,draw,inner sep=0pt,minimum size=3pt]
\tikzstyle{bvartex}=[circle,fill=black,draw,inner sep=0pt,minimum size=3pt]

\numberwithin{equation}{section}
\numberwithin{figure}{section}

\theoremstyle{plain}
\newtheorem{Thm}{Theorem}[section]

\newtheorem{Lemma}[Thm]{Lemma}
\newtheorem{Cor}[Thm]{Corollary}

\newtheorem*{Claim}{Claim}

\theoremstyle{definition}
\newtheorem{Defn}[Thm]{Definition}

\newtheorem{Ques}[Thm]{Question}
\newtheorem*{Def*}{Definition}

\theoremstyle{remark}
\newtheorem*{Remark}{Remark}
\newtheorem*{Remarks}{Remarks}

\let \seq=\subseteq

\let \To=\Rightarrow

\let \0=\emptyset
\let \1=\infty
\let \al=\alpha
\let \be=\beta
\let \ga=\gamma

\let \De=\Delta

\let \si=\sigma
\let \Si=\Sigma

\let \f=\varphi
\let \p=\ldots
\let \h=\text

\let \cd=\cdots

\let \cov=\vartriangleleft
\let \voc=\vartriangleright

\let \i=\iota
\let \t=\theta

\let \pa=\partial

\newcommand{\WLOGtwo}{without loss of generality}
\newcommand{\Th}{Theorem}
\newcommand{\Ths}{Theorems}
\newcommand{\Co}{Corollary}

\newcommand{\Pro}{Proposition}

\newcommand{\Def}{Definition}

\newcommand{\Exa}{Example}

\newcommand{\Con}{Conjecture}

\newcommand{\Z}{\mathbf{Z}}

\newcommand{\ol}{\overline}

\newcommand{\id}{\mathrm{id}}

\DeclareMathOperator{\rk}{rk}

\newcommand{\Br}{\mathrm{Br}}
\newcommand{\J}{\mathfrak{I}}
\newcommand{\twist}{\J(\t)}
\newcommand{\twid}{\i(\t)}
\newcommand{\xy}{x'}
\newcommand{\dx}{\De_{xyz}}
\newcommand{\ddx}{\De_{\xy}'}
\newcommand{\deldx}{\del_{\dx}(\{x,y,z\})}
\newcommand{\delddx}{\del_{\ddx}(\{\xy\})}
\newcommand{\xh}{(x,\hat1)}
\newcommand{\yh}{(y,\hat1)}
\newcommand{\zh}{(z,\hat1)}
\newcommand{\dxh}{\De\xh}
\newcommand{\dyh}{\De\yh}
\newcommand{\dzh}{\De\zh}
\newcommand{\dxy}{\De(x,y)}
\newcommand{\join}{*}
\newcommand{\2}{\mathbf{2}}
\newcommand{\fib}{\mathcal{F}}
\newcommand{\hgt}{\mathrm{ht}}

\DeclareMathOperator{\cl}{cl}

\DeclareMathOperator{\lk}{lk}
\DeclareMathOperator{\del}{del}

\title{Topology of posets with special partial matchings}
\author{Nancy Abdallah \and Mikael Hansson \and Axel Hultman}
\thanks{}
\address{Department of Mathematics, Link\"oping University, SE-581 83 Link\"oping, Sweden}
\email{nancy.abdallah@liu.se}
\email{mikael.hansson@liu.se}
\email{axel.hultman@liu.se}

\begin{document}

\begin{abstract}
Special partial matchings (SPMs) are a generalisation of Brenti's special matchings.
Let a \emph{pircon} be a poset in which every non-trivial principal order ideal is finite and admits an SPM. Thus pircons generalise Marietti's zircons.
We prove that every open interval in a pircon is a PL ball or a PL sphere.
It is then demonstrated that Bruhat orders on certain twisted identities and quasiparabolic $W$-sets constitute pircons.
Together, these results extend a result of Can, Cherniavsky, and Twelbeck, prove a conjecture of Hultman, and confirm a claim of Rains and Vazirani.
\end{abstract}

\maketitle

\section{Introduction} \label{se:intro}

A special matching on a poset is a complete matching of the Hasse diagram satisfying certain extra conditions.
The concept was introduced by Brenti~\cite{Brenti_intersection}.
For eulerian posets, an equivalent notion was also independently introduced by du Cloux~\cite{du_Cloux}.
Their main motivation was to provide an abstract framework in which to study the Bruhat order on a Coxeter group.
Namely, every non-trivial lower interval in the Bruhat order admits a special matching.
Thus, Bruhat orders provide examples of \emph{zircons}, posets in which every non-trivial principal order ideal is finite and has a special matching.
Beginning with Marietti~\cite{Marietti}, zircons have been the focal point of a lot of attention; see, e.g., \cite{C-M,Hultman_zircon,Marietti_coxeter}.
Notably, (the order complex of) any open interval in a zircon is a PL sphere;
this is essentially a result of du Cloux~\cite[\Co~3.6]{du_Cloux}, which is based on results from Dyer's thesis~\cite{Dyer_thesis}.
Reading~\cite{Reading} provided a different proof.\footnote{Although Reading worked in the context of Bruhat orders, his proof is valid in the more general zircon setting.}

In \cite{A-H}, two of the present authors generalised the special matching concept to special partial matchings (SPMs), which are not necessarily complete matchings satisfying similar conditions.
Generalising zircons, let us say that a \emph{pircon} is a poset in which every non-trivial principal order ideal is finite and admits an SPM.
These notions, too, are originally motivated by Coxeter group theory:
the dual of the Bruhat order on the fixed point free involutions in the symmetric group is a pircon~\cite{A-H}.
This is generalised considerably in Section~\ref{se:coxeter}, where it is demonstrated that the Bruhat order on the twisted identities $\twid$ is a pircon whenever the involution $\t$ has the so-called NOF property.
Moreover, Bruhat orders on Rains and Vazirani's~\cite{R-V} quasiparabolic $W$-sets (under a boundedness assumption) form pircons.
In particular, this applies to all parabolic quotients of Coxeter groups.

We investigate the topology of posets with SPMs.
Our first main result roughly states that an SPM provides a way to ``lift'' the PL ball or sphere property from a subinterval; this is \Th~\ref{th:main}.
It follows that every open interval in a pircon is a PL ball or a PL sphere, which is our second main result.
In particular, this proves a conjecture from \cite{Hultman3} on Bruhat orders on twisted identities, and confirms a claim from \cite{R-V} about quasiparabolic $W$-sets.

The overall proof strategy is inspired by that of Reading's aforementioned proof in \cite{Reading}.
Roughly, if $P$ is a poset with minimum $\hat0$, maximum $\hat1$, and an SPM $M$,
we prove that $P$ can be obtained from the interval $[\hat0,M(\hat1)]$ using certain modifications.
Investigating the effect of these modifications on the poset topology forms the technical backbone of the paper.

The remainder of the paper is structured in the following way.
In the next section, we recall basic definitions and review some useful results from the literature.
Then, in Section~\ref{se:PL_tools}, we prove a couple of elementary lemmas that later serve as the main topological tools.
In Section~\ref{se:zip}, ways to locally modify posets, including a version of Reading's ``zippings'' from \cite{Reading}, are studied.
It is shown that these modifications preserve the PL ball or sphere property.
After that, in Section~\ref{se:spm}, we recall the definition of an SPM and prove that a poset which admits an SPM can be obtained from one which in some sense is easier to understand, using the modifications studied in the previous section.
Combining the results of the previous two sections, the main results follow essentially at once; this is the content of Section~\ref{se:main}.
In Section~\ref{se:coxeter}, we explain how examples of pircons are provided by Bruhat orders, first on twisted identities and second on quasiparabolic $W$-sets in Coxeter groups. The implications of our second main result in these contexts are discussed. Finally, in the last section, we raise some open questions.

\section{Preliminaries} \label{se:prel}

In this section, preliminary material on posets (partially ordered sets) and topology of simplicial complexes is gathered.

\subsection{Posets} \label{sse:posets}

Let $P$ be a poset.
If $P$ contains an element denoted $\hat0$ or $\hat1$, it is assumed to be a minimum or a maximum, respectively, i.e., $x \geq \hat0$ and $x \leq \hat1$ for all $x \in P$.
The \emph{proper part} of $P$ is then $\ol{P}=P-\{\hat0,\hat1\}$.

Standard interval notation is employed for posets. Thus, if $x,y \in P$, then
\[
[x,y]=\{z \in P \mid x \leq z \leq y\},
\]
with the induced order from $P$, and similarly for open and half-open intervals.

An \emph{order ideal} $J \seq P$ is an induced subposet closed under going down, i.e., $x \leq y \in J \To x \in J$.
The complement of an order ideal is called an \emph{order filter}. An order ideal is \emph{principal} if it has a maximum.
For principal order ideals, the notation $P_{\leq y}=\{x \in P \mid x \leq y\}$ is convenient.
Similarly, $P_{<y}$, $P_{\geq y}$, and $P_{>y}$ are defined in the obvious way.

Suppose every principal order ideal in $P$ is finite. If, for any $y \in P$, all maximal chains (totally ordered subsets) in $P_{\leq y}$ have the same number of elements, $P$ is called \emph{graded}.
In this case, there is a unique \emph{rank function}, i.e., a function $\rk: P \to \{0,1,\p\}$ such that $\rk(x)=0$ if $x$ is minimal, and $\rk(y)=\rk(x)+1$ if $y$ covers $x$.

Suppose $\pi: P \to P'$ is an order-preserving map of posets.
Then $\pi$ is called an \emph{order projection} if for every ordered pair $x' \leq_{P'} y'$ in $P'$ there exist $x \leq_P y$ in $P$ such that $\pi(x)=x'$ and $\pi(y)=y'$.
In particular, any order projection is surjective. We construct the quotient $\fib_\pi$ as follows.
The elements of $\fib_\pi$ are the fibres $\pi^{-1}(x')=\{x \in P \mid \pi(x)=x'\}$ for $x' \in P'$.
A relation on $\fib_\pi$ is given by $F_1 \leq_{\fib_\pi} F_2$ if $x \leq_P y$ for some $x \in F_1$ and $y \in F_2$.
This is a partial order if $\pi$ is an order projection. We then call $\fib_\pi$ the \emph{fibre poset}. It is isomorphic to $P'$:

\begin{Lemma}[\mbox{\cite[\Pro~1.1]{Reading}}] \label{le:fibre}
If $\pi: P \to P'$ is an order projection, then $\fib_\pi$ and $P'$ are isomorphic posets.
\end{Lemma}

\subsection{Simplicial complexes} \label{sse:complexes}

\emph{Throughout the present paper, all simplicial complexes are finite.} By convention, the empty set is considered to be a simplex of every non-void simplicial complex.
Given an (abstract) simplicial complex $\De$, we shall denote its geometric realisation (defined up to linear homeomorphism) by $\|\De\|$, a polyhedron in some real euclidean space. The simplices of $\De$ are sometimes called its \emph{faces}, and maximal faces are referred to as \emph{facets}.

For a face $\si \in \De$, the subcomplex
\[
\lk_\De(\si)=\{\tau \in \De \mid \h{$\si \cap \tau=\0$ and $\si \cup \tau \in \De$}\}
\]
is the \emph{link} of $\si$.

If $V$ is a set of vertices of $\De$, the \emph{deletion} of $V$ in $\De$ is the subcomplex
\[
\del_\De(V)=\{\si \in \De \mid \si \cap V=\0\}.
\]

The \emph{join} $\De \join \De'$ of two simplicial complexes $\De$ and $\De'$ is a new simplicial complex defined (up to isomorphism) as follows.
Suppose the vertex sets of $\De$ and $\De'$ are disjoint (otherwise, first replace $\De'$, say, by a suitable isomorphic copy), and let
\[
\De \join \De'=\{\si \cup \tau \mid \h{$\si \in \De$ and $\tau \in \De'$}\}.
\]

If $\mathcal{F}$ is a finite family of finite sets, $\cl(\mathcal{F})$ denotes the simplicial complex generated by $\mathcal{F}$, i.e.,
\[
\cl(\mathcal{F})=\{\si \mid \h{$\si \seq F$ for some $F \in \mathcal{F}$}\};
\]
it is called the \emph{closure} of $\mathcal{F}$.

Let $\si \prec \tau$ indicate that $\si \subset \tau$ and $\dim\si=\dim\tau-1$. If $\si \prec \tau$ and $\tau$ is the unique face (necessarily a facet) of $\De$ which properly contains $\si$, then the modification $\De \searrow \De-\{\si,\tau\}$ is an \emph{elementary collapse}.
A simplicial complex $\De$ is \emph{collapsible} if $\De \searrow \cd \searrow \0$.
Forman's discrete Morse theory~\cite{Forman} provides a convenient method to establish collapsibility.
The formulation in terms of matchings which we use here is due to Chari~\cite{Chari}; see also Forman~\cite{Forman2}.

A \emph{complete matching} on $\De$ is a function $\mu: \De \to \De$ which satisfies $\mu^2=\id$ and either $\si \prec \mu(\si)$ or $\mu(\si) \prec \si$ for all $\si \in \De$.
Then $\mu$ is \emph{acyclic} if
\[
\si_0 \prec \mu(\si_0) \succ \si_1 \prec \mu(\si_1) \succ \cd \prec \mu(\si_{t-1}) \succ \si_t
\]
with $\si_0 \neq \si_1$ implies that $\si_t \neq \si_0$.

\begin{Lemma}[\mbox{Forman~\cite{Forman}}] \label{le:forman}
A simplicial complex is collapsible if it has an acyclic complete matching.
\end{Lemma}

Given a finite poset $P$, its \emph{order complex} $\De(P)$ is the simplicial complex whose faces are the chains in $P$.
In order to prevent proliferation of brackets when taking order complexes of poset intervals, we shall write $\dxy$ instead of $\De((x,y))$, $\De[x,y)$ instead of $\De([x,y))$, and so on.

\subsection{PL topology} \label{sse:topology}

Next, some notions from PL topology are reviewed. We refer to, e.g., \cite{Hudson} or \cite{Rourke-Sanderson} for this and much more information.

Suppose $\De$ and $\De'$ are simplicial complexes.
A continuous map $f: \|\De\| \to \|\De'\|$ is \emph{piecewise linear}, or \emph{PL}, if its graph is a euclidean polyhedron.
This is equivalent to there being simplicial subdivisions $\tilde{\De}$ and $\tilde{\De}'$ of $\De$ and $\De'$, respectively, with respect to which $f$ is a simplicial map of the corresponding triangulations of $\|\De\|$ and $\|\De'\|$.

Say that $\De$ and $\De'$ are \emph{PL homeomorphic} if there exists a PL homeomorphism $f: \|\De\| \to \|\De'\|$ (it follows that $f^{-1}$ is also PL).

A \emph{PL $d$-ball} is a simplicial complex which is PL homeomorphic to the simplicial complex $\De^d$ whose only facet is the $d$-dimensional simplex.
A \emph{PL $(d-1)$-sphere} is a simplicial complex which is PL homeomorphic to the simplicial complex obtained by removing the facet from $\De^d$.
In the PL category, balls and spheres behave as expected with respect to joins:

\begin{Lemma}[\mbox{\cite[Lemma~1.13]{Hudson}}] \label{le:join}
Let $B^d$ denote a PL $d$-ball and $S^d$ a PL $d$-sphere.
Then $B^k \join B^l \cong B^k \join S^l \cong B^{k+l+1}$ and $S^k \join S^l \cong S^{k+l+1}$, where $\cong$ means PL homeomorphic.
\end{Lemma}

In particular, the cone over a PL $d$-ball or a PL $d$-sphere is a PL $(d+1)$-ball, since a cone is a join with the $0$-ball.

A \emph{PL $d$-manifold} is a simplicial complex satisfying that, for all $k \geq 0$, the link of every $k$-dimensional face is a PL $(d-1-k)$-ball or sphere.
If $\De$ is a PL $d$-manifold, its \emph{boundary} $\pa\De$ is the simplicial complex whose facets are the $(d-1)$-dimensional faces of $\De$ that are contained in only one facet of $\De$.
PL $d$-balls are PL $d$-manifolds with PL $(d-1)$-spheres as boundaries. PL $d$-spheres are PL $d$-manifolds without boundaries.

If $P$ is a finite poset with $\hat0$ and $\hat1$, every link in the order complex $\De(\ol{P})$ is a join of order complexes of open intervals in $P$.
Hence, by Lemma~\ref{le:join}, $\De(\ol{P})$ is a PL manifold if and only if $P$ is graded and $\dxy$ is a PL ball or sphere for every interval $(x,y) \neq (\hat0,\hat1)$ in $P$.

As we shall see, the next lemma opens up for inductive arguments. However plausible it seems, the first statement would be false without the PL condition.

\begin{Lemma} \label{le:PL}
${}$
\begin{itemize}
  \item[(i)]  If $\De_1$ and $\De_2$ are PL $d$-balls and $\De_1 \cap \De_2$ is a PL $(d-1)$-ball contained in $\pa\De_1 \cap \pa\De_2$, then $\De_1 \cup \De_2$ is a PL $d$-ball.
  \item[(ii)] If $\De_1$ and $\De_2$ are PL $d$-balls with $\De_1 \cap \De_2=\pa\De_1=\pa\De_2$, then $\De_1 \cup \De_2$ is a PL $d$-sphere.
\end{itemize}
\end{Lemma}

For a proof of (i), see \cite[\Co~1.28]{Hudson}. A proof of (ii) can be found in \cite{Mandel_thesis}.

Although the second sentence of the following result is rarely stated explicitly, it follows from, e.g., the first part of Hudson's proof; see \cite[\Th~1.26]{Hudson}.

\begin{Lemma}[Newman's theorem] \label{le:newman}
The closure of the complement of a PL $d$-ball embedded in a PL $d$-sphere is a PL $d$-ball. Moreover, the two balls have the same boundary.
\end{Lemma}

In particular, the deletion of a single vertex $v$ in a PL $d$-sphere is a PL $d$-ball, since it is the closure of the complement of a cone over the link of $\{v\}$.

\begin{Lemma}[\mbox{\cite[\Co~1.27]{Hudson}}] \label{Hudson}
If $A$ is a PL $d$-ball and $F$ is a PL $(d-1)$-ball contained in $\pa{A}$,
then any PL homeomorphism $\|F\| \to \|\De^{d-1}\|$ extends to a PL homeomorphism $\|A\| \to \|\De^d\|$.
\end{Lemma}

\begin{Lemma}[\mbox{Whitehead~\cite{Whitehead}; see also \cite[\Co~3.28]{Rourke-Sanderson}}] \label{le:whitehead}
A collapsible PL manifold is a PL ball.
\end{Lemma}

\section{PL topological tools} \label{se:PL_tools}

In this section, we develop elementary PL topological machinery that will serve as our toolbox in the proofs of the main results.

Let $\2$ denote the totally ordered, two-element poset $\{\al,\be\}$ where $\al<\be$.

\begin{Lemma} \label{le:collapsible}
If $P$ is a finite poset with $\hat0$ and $\hat1$, then $\De(\ol{P \times \2}-\{(\hat0,\be)\})$ is collapsible.
\end{Lemma}

\begin{proof}
We shall apply Lemma~\ref{le:forman}. For brevity, let $Q=\ol{P \times \2}-\{(\hat0,\be)\}$. Given a chain
\[
C=\{(x_1,\ga_1) < (x_2,\ga_2) <\cd <(x_m,\ga_m)\} \seq Q,
\]
put $(x_{m+1},\ga_{m+1})=(\hat1,\be)$, and let $j$ be the smallest index such that $\ga_j=\be$. Define $p(C)=(x_j,\al)$.
Observe that $C \cup \{p(C)\}$ is a chain in $Q$, and that $p(C \cup \{p(C)\})=p(C)=p(C-\{p(C)\})$. Therefore,
\[
\mu(C)=\begin{cases}C \cup \{p(C)\} & \h{if $p(C) \notin C$,} \\ C-\{p(C)\} & \h{otherwise}\end{cases}
\]
defines a complete matching $\mu$ on $\De(Q)$.
Now, if $C_0 \prec \mu(C_0) \succ C_1 \prec \mu(C_1)$ for chains $C_0 \neq C_1$, then $C_1$ has fewer elements than $C_0$ with $\be$ as the second component.
Hence $\mu$ is acyclic.
\end{proof}

\begin{Lemma} \label{le:extension}
Suppose $P$ is a finite poset with $\hat0$ and $\hat1$.
If $\De(\ol{P})$ is a PL $d$-ball (a PL $d$-sphere), then $\De(\ol{P \times \2})$ is a PL $(d+1)$-ball (a PL $(d+1)$-sphere).
In either case, $\De(\ol{P \times \2}-\{(\hat0,\be)\})$ is a PL $(d+1)$-ball.
\end{Lemma}

\begin{proof}
Let $R=\ol{P \times \2}$ and $Q=R-\{(\hat0,\be)\}$. We induct on $d$, all assertions being clear when $d=0$.

For $p \in \ol{P}$, we have the following two poset isomorphisms:
\[
Q_{<(p,\ga)} \cong \begin{cases}\ol{P_{\leq p}} & \h{if $\ga=\al$,} \\ \ol{P_{\leq p} \times \2}-\{(\hat0,\be)\} & \h{if $\ga=\be$,}\end{cases}
\]
and
\[
Q_{>(p,\ga)} \cong \begin{cases}\ol{P_{\geq p} \times \2} & \h{if $\ga=\al$,} \\ \ol{P_{\geq p}} & \h{if $\ga=\be$.}\end{cases}
\]
Moreover, $Q_{<(1,\al)} \cong \ol{P}$. The induction assumption therefore implies that all links of non-empty faces in $\De(Q)$ are PL balls or spheres.
Hence $\De(Q)$ is a PL $(d+1)$-manifold. Now Lemmas~\ref{le:whitehead} and \ref{le:collapsible} imply that $\De(Q)$ is a PL $(d+1)$-ball.

Next observe that
\[
\De(R)=\De(Q) \cup \De\Big(R_{\geq (\hat0,\be)}\Big).
\]
Both complexes in the union are PL $(d+1)$-balls; the latter is isomorphic to a cone over $\De(\ol{P})$. Furthermore, we have
\[
\De(Q) \cap \De\Big(R_{\geq (\hat0,\be)}\Big)=\De\Big(R_{> (\hat0,\be)}\Big),
\]
which is contained in the boundary of both balls. On the other hand, this intersection is isomorphic to $\De(\ol{P})$.
The desired conclusions about $\De(R)$ now follow from Lemma~\ref{le:PL}.
\end{proof}

We shall frequently find the need to modify simplicial complexes by replacing balls with other balls.
The following two statements describe circumstances under which the topology is left unchanged.

\begin{Lemma} \label{ball}
Suppose $\De$, $A$, and $A'$ are PL $d$-balls such that $A \seq \De$ and $A' \cap \De=\pa{A'}=\pa{A}$. Then $(\De-A) \cup A'$ is a PL $d$-ball.
\end{Lemma}

\begin{proof}
Let $C$ be a cone over $\pa\De$ whose apex $v$ is disjoint from $A'$ and $\De$. By Lemma~\ref{le:PL}(ii), $S=\De \cup C$ is a PL $d$-sphere.
Put $a=\cl(S-A)$, which is a PL $d$-ball with $\pa{a}=\pa{A}$ by Lemma~\ref{le:newman}.
Since $A' \cap \De=\pa{A'}=\pa{A}$, $\Si=a \cup A'$ is a PL $d$-sphere by Lemma~\ref{le:PL}(ii). Hence,
\[
\del_\Si(\{v\})=(\De-A) \cup A'
\]
is a PL $d$-ball.
\end{proof}

\begin{Lemma} \label{homeo}
Let $\De$ be a simplicial complex. Suppose $A$ and $A'$ are PL $d$-balls and $F$ is a PL $(d-1)$-ball such that $A \seq \De$ and $F \seq \pa{A} \cap \pa{A'}$.
If both $\cl(\De-A) \cap A$ and $\cl(\De-A) \cap A'$ are contained in $F$, then $\De$ and $(\De-A) \cup A'$ are PL homeomorphic.
\end{Lemma}

\begin{proof}
There is a PL homeomorphism $\f: \|F\| \to \|\De^{d-1}\|$.
By Lemma~\ref{Hudson}, it extends to PL homeomorphisms $\f_1: \|A\| \to \|\De^d\|$ and $\f_2: \|A'\| \to \|\De^d\|$. Let $\psi=\f_2^{-1} \circ \f_1$.
Then $\psi: \|A\| \to \|A'\|$ is a PL homeomorphism whose restriction to $\|F\|$ is the identity map. Obviously, $\De=\cl(\De-A)\cup A$.
Moreover, $(\De-A) \cup A'=\cl(\De-A) \cup A'$ because $\cl(\De-A) \seq (\De-A) \cup F$. Now define $f: \|\De\| \to \|(\De-A) \cup A'\|$ by
\[
f(x)=\begin{cases}\psi(x) & \h{if $x \in \|A\|$,} \\ x & \h{if $x \in \|\cl(\De-A)\|$.}\end{cases}
\]
Then $f$ is a well-defined PL map because $\cl(\De-A)\cap A \seq F$, and the same holds for $f^{-1}$ since $\cl(\De-A)\cap A' \seq F$.
\end{proof}

\section{Zippings and removals} \label{se:zip}

In \cite{Reading}, Reading introduced the concept of a zipper in a poset. We restrict his definition somewhat.

\begin{Defn}[Reading~\cite{Reading}]
Let $P$ be a finite poset with $\hat0$ and $\hat1$, and distinct elements $x,y,z \in P$. Call $(x,y,z)$ a \emph{zipper} if
\begin{itemize}
  \item[(i)]   $z$ covers only $x$ and $y$,
  \item[(ii)]  $z=x \vee y$, where $\vee$ denotes join (supremum), and
  \item[(iii)] $[\hat0,x)=[\hat0,y)$.
\end{itemize}
The zipper is \emph{proper} if $z \neq \hat1$.
\end{Defn}

\begin{Defn}[Reading~\cite{Reading}]
Given $P$ with a partial order $\leq$ and a proper zipper $(x,y,z)$, let $P'=(P-\{x,y,z\}) \biguplus \{\xy\}$, and define a partial order $\leq'$ on $P'$ by
\begin{itemize}
  \item $a \leq' b$ if $a \leq b$,
  \item $\xy \leq' a$ if $x \leq a$ or $y \leq a$,
  \item $a \leq' \xy$ if $a \leq x$ (or, equivalently, $a \leq y$), and
  \item $\xy \leq' \xy$.
\end{itemize}
\end{Defn}

The fact that $\leq'$ is a partial order on $P'$ is \cite[\Pro~4.1]{Reading}. We say that $P'$ is the result of a \emph{zipping} in $P$.
The effect is that $P'$ is obtained from $P$ by identifying the elements $x$, $y$, and $z$; they become the element $\xy$.
Reading proved that this preserves PL spheres:

\begin{Thm}[\mbox{\cite[\Th~4.7]{Reading}}] \label{4.7}
If $P'$ is obtained from $P$ by zipping a proper zipper and $\De(\ol{P})$ is a PL $d$-sphere, then so is $\De(\ol{P'})$.
\end{Thm}

We shall prove a similar result for PL balls. In contrast to spheres, balls have boundaries.
This causes complications that can be overcome by imposing additional restrictions on zippers.
A version which suffices for our needs is the content of the next definition.

Recall that a \emph{coatom} in a poset with $\hat1$ is an element covered by $\hat1$.

\begin{Defn} \label{de:clean}
A zipper $(x,y,z)$ is \emph{clean} if it is proper, and for some coatom $c$ there exists a poset isomorphism $\f: [x,\hat1] \to [x,c] \times \2$ such that $\f(z)=(x,\be)$.
\end{Defn}

\begin{Thm} \label{MT1}
If $P'$ is obtained from $P$ by zipping a clean zipper and $\De(\ol{P})$ is a PL $d$-ball, then so is $\De(\ol{P'})$.
\end{Thm}

\begin{proof}
Suppose $\De(\ol{P})$ is a PL $d$-ball and $(x,y,z)$ is a clean zipper in $P$.
Let $\dx$ be the simplicial complex whose facets are the maximal chains in $\ol{P}$ containing $x$ or $y$ (note that this includes all that contain $z$), and let $\ddx$ be the simplicial complex whose facets are the maximal chains in $\ol{P'}$ containing $\xy$.
By the definition of a zipping, $\De(\ol{P}-\{x,y,z\})=\De(\ol{P'}-\{\xy\})$ and $\deldx=\delddx$. Hence,
\begin{align} \label{eq:ball_exchange}
\De(\ol{P'}) &= \De(\ol{P'}-\{\xy\}) \cup \ddx \\
             &= \De(\ol{P}-\{x,y,z\}) \cup \ddx \nonumber \\
             &= (\De(\ol{P})-\dx)\cup \deldx \cup \ddx \nonumber \\
             &= (\De(\ol{P})-\dx)\cup \delddx \cup \ddx \nonumber \\
             &= (\De(\ol{P})-\dx)\cup \ddx. \nonumber
\end{align}
That is, $\De(\ol{P'})$ is obtained from $\De(\ol{P})$ by removing $\dx$ and inserting $\ddx$.
Our goal is to apply either Lemma~\ref{ball} or Lemma~\ref{homeo} with $\De=\De(\ol{P})$, $A=\dx$, $A'=\ddx$, and (if needed) $F=\deldx=\delddx$.
The hypotheses must be verified.

Even though it originally concerns the situation when $\De(\ol{P})$ is a sphere, the appropriate part of Reading's proof of \cite[\Th~4.7]{Reading} shows that $\dx$ is a PL $d$-ball also in our situation.\footnote{One invokes Lemma~\ref{le:PL}(i) using that $\dx$ is the union of the PL $d$-balls $\De(\hat0,x] \join \dxh$ and $\De(\hat0,y] \join \dyh$ whose intersection is the PL $(d-1)$-ball $\De(\hat0,x) \join \De[z,\hat1)$ which is contained in the boundary of both.}

Next we observe that $\deldx \seq \pa\dx$.
Indeed, since $z=x \vee y$, the cleanness of $(x,y,z)$ implies that every facet $C$ in $\deldx$ contains some $w$ which covers exactly one of $x$ and $y$, say $x$.
Hence, $C$ extends uniquely to a facet in $\dx$, namely by adding $x$.

\begin{Claim}
If $\dyh$ and $\De(\hat0,x)$ are PL spheres, $\deldx$ is a PL $(d-1)$-sphere. Otherwise, $\deldx$ is a PL $(d-1)$-ball.
\end{Claim}

Let us assume this claim for now and turn to its proof later.

Suppose first that $\deldx$ is a sphere. Since it cannot be a proper subcomplex of another $(d-1)$-sphere, $\deldx=\pa\dx$.
Since $\ddx$ is a cone over the PL sphere $\delddx=\deldx$ with apex $\xy$, $\ddx$ is a PL $d$-ball and $\deldx=\pa\ddx$.
By Lemma~\ref{ball} and \eqref{eq:ball_exchange}, $\De(\ol{P'})$ is a PL $d$-ball.

Now suppose $\deldx$ is a ball. Since $\ddx$ is a cone over this ball with apex $\xy$, $\ddx$ is a PL $d$-ball with the PL $(d-1)$-ball $\delddx=\deldx$ contained in its boundary. Observe that
\[
\cl(\De(\ol{P})-\dx) \cap \dx \seq \De(\ol{P}-\{x,y,z\}) \cap \dx=\deldx
\]
and
\[
\cl(\De(\ol{P})-\dx) \cap \ddx \seq \De(\ol{P'}-\{\xy\}) \cap \ddx=\delddx.
\]
Lemma~\ref{homeo} now shows that $\De(\ol{P})$ and $(\De(\ol{P})-\dx) \cup \ddx$ are PL homeomorphic. By \eqref{eq:ball_exchange}, $\De(\ol{P'})$ is a PL $d$-ball.

\medskip

It remains to verify the claim. Define $\De=\dyh$, $A=\De[z,\hat1)$, and $A'=\De(\xh-\{z\})$. Observe that
\[
\deldx=\De(\hat0,x) \join ((\De-A)\cup A').
\]
By Lemma~\ref{le:join}, the claim follows if $\De$ and $(\De-A)\cup A'$ are PL homeomorphic. There are two cases:

\textbf{Case~1: $\De$ is a PL $k$-sphere.} In this case, $\dzh$ and $\dxh$ are spheres, the former because it is a link in $\dyh$, the latter by Lemma~\ref{le:extension} because $[x,c] \cong [z,\hat1]$ with $c$ being the coatom of \Def~\ref{de:clean}.
Hence, $A$ and $A'$ are PL $k$-balls, and $\pa{A}=\pa{A'}=\dzh=\De \cap A'$.
By Lemma~\ref{le:newman}, $\cl(\De-A)$ is a PL $k$-ball, and thus Lemma~\ref{le:PL}(ii) implies that $\cl(\De-A) \cup A'=(\De-A) \cup A'$ is a PL $k$-sphere, as desired.

\textbf{Case~2: $\De$ is a PL $k$-ball.} We shall apply Lemma~\ref{ball} or Lemma~\ref{homeo}, the latter with $F=\dzh$.
Again, there is a coatom $c$ such that $[x,c] \cong [z,\hat1]$.
By Lemma~\ref{le:extension}, $A'$ is a PL $k$-ball, as are $A$ and $\De$, whereas $F$ is either a PL $(k-1)$-ball or a PL $(k-1)$-sphere.
Since $A$ is a cone over $F$, $F \seq \pa{A}$. Consider a maximal chain $C$ in $\zh$ with minimum $w$ (let $w=\hat1$ if $\zh$ is empty).
Then $\f(w)=(v,\be)$ for some $v \leq c$ which covers $x$, where $\f: [x,\hat1] \to [x,c] \times \2$ is the poset isomorphism provided by \Def~\ref{de:clean}.
The only way to extend $C$ to a maximal chain in $\xh-\{z\}$ is to add $v$. Hence $F \seq \pa{A'}$.

If $F$ is a sphere, we have $F=\pa{A}=\pa{A'}$ since a sphere cannot be a proper subcomplex of another sphere of the same dimension.
Lemma~\ref{ball} then shows that $(\De-A)\cup A'$ is a PL $k$-ball.

If, instead, $F$ is a ball, we observe that
\[
\cl(\De-A) \cap A \seq \De(\yh-\{z\}) \cap \De[z,\hat1)=\dzh=F
\]
and
\[
\cl(\De-A) \cap A' \seq \De(\yh-\{z\}) \cap \De(\xh-\{z\})=\dzh=F.
\]
Thus, Lemma~\ref{homeo} implies that $(\De-A) \cup A'$ is a PL $k$-ball. The claim is established.
\end{proof}

In addition to zippings, we shall find the need for another way to modify posets which also preserves PL balls.

\begin{Defn} \label{de:removable}
Let $P$ be a finite poset with $\hat0$ and $\hat1$.
An element $z \neq \hat1$ is called \emph{removable} if $z$ covers exactly one element $x$, and for some coatom $c$ there exists a poset isomorphism $\f: [x,\hat1] \to [x,c] \times \2$ such that $\f(z)=(x,\be)$.
\end{Defn}

If $z \in P$ is removable, we shall refer to $P-\{z\}$ as obtained by a \emph{removal}.
Alternatively, in analogy with zippings, we may consider $P-\{z\}$ as being obtained by identifying $x$ and $z$.
Removals produce balls from PL balls or spheres:

\begin{Thm} \label{MT2}
Suppose $z \in P$ is removable. If $\De(\ol{P})$ is a PL $d$-ball or a PL $d$-sphere, then $\De(\ol{P}-\{z\})$ is a PL $d$-ball.
\end{Thm}

\begin{proof}
Let $x$ and $c$ be as in \Def~\ref{de:removable}. Since $\De(x,c)$ is a PL ball or sphere, $\De(\xh-\{z\})$ is a PL ball by Lemma~\ref{le:extension}.
If $x=\hat0$ we are done, so suppose $x>\hat0$.
Then $\De(\ol{P})$ is a ball since the link of $\{z\}$ is a cone with apex $x$ and therefore not a sphere.
Let $\De_x$ be the simplicial complex whose facets are maximal chains in $\ol{P}$ containing $x$.
We shall apply Lemma~\ref{homeo} with $\De=\De(\ol{P})$, $A=\De_x$, $A'=\del_{\De_x}(\{z\})$, and $F=\del_{\De_x}(\{x,z\})$.
Since $\De_x$ is a cone over \smash{$\lk_{\De(\ol{P})}(\{x\})$}, $A$ is a PL $d$-ball satisfying $F \seq \pa{A}$.
Furthermore, $F=\De(\hat0,x) \join \De(\xh-\{z\})$, which is a PL $(d-1)$-ball by Lemma~\ref{le:join}, and $A'$ is a cone over $F$, hence a PL $d$-ball with $F$ in its boundary.
Finally, $\cl(\De-A) \cap A' \seq \cl(\De-A) \cap A \seq F$ because every chain which contains $z$ or $x$ is contained in $\De_x$.
By Lemma~\ref{homeo},
\[
(\De-A) \cup A'=(\De(\ol{P})-\De_x) \cup \del_{\De_x}(\{z\})=\De(\ol{P}-\{z\})
\]
is PL homeomorphic to $\De=\De(\ol{P})$.
\end{proof}

\section{Special partial matchings} \label{se:spm}

The following definition is taken from \cite{A-H}.

\begin{Defn}\label{de:spm}
Suppose $P$ is a finite poset with $\hat1$, and let $\cov$ denote its cover relation.
A \emph{special partial matching}, or \emph{SPM}, on $P$ is a function $M: P \to P$ such that
\begin{itemize}
  \item $M^2=\id$,
  \item $M(\hat1) \cov \hat1$,
  \item for all $x\in P$, we have $M(x) \cov x$, $M(x)=x$, or $x\cov M(x)$, and
  \item if $x\cov y$ and $M(x) \neq y$, then $M(x)<M(y)$.
\end{itemize}
\end{Defn}

The terminology comes from the fact that an SPM without fixed points is precisely a \emph{special matching} as defined by Brenti~\cite{Brenti_intersection}.

For special matchings, the following important lemma is essentially due to Brenti;
see \cite[Lemma~4.2]{Brenti_intersection}, which is, however, stated under a gradedness assumption.
A proof without this assumption appears in \cite{Hultman_zircon}. We provide here a different proof which is valid also for SPMs.

\begin{Lemma}[Lifting property] \label{le:lifting}
Suppose that $P$ is a finite poset with $\hat1$, and $M$ is an SPM on $P$. If $x,y \in P$ with $x<y$ and $M(y) \leq y$, then
\begin{itemize}
  \item[(i)]   $M(x) \leq y$,
  \item[(ii)]  $M(x) \leq x \To M(x)<M(y)$, and
  \item[(iii)] $M(x) \geq x \To x \leq M(y)$.
\end{itemize}
\end{Lemma}

\begin{proof}
It suffices to prove (i) and (ii) because together they imply (iii).

Consider a saturated chain $x=x_0 \cov x_1 \cov \cd \cov x_k=y$. By the definition of an SPM, for each $i<k$, either $M(x_i)<M(x_{i+1})$ or $M(x_i)=x_{i+1}$.

(i) We either have $M(x_0)<M(x_1)<\cd<M(y) \leq y$ or $M(x_0)<M(x_1)<\cd<M(x_i)=x_{i+1} \leq y$ for some $i<k$.

(ii) We either have $M(x_0)<M(x_1)<\cd<M(y)$ or $M(y)>M(x_{k-1})>\cd>M(x_{i+1})=x_i \geq x \geq M(x)$ for some $i<k$.
\end{proof}

Next, a fundamental construction is described.
It presents a poset with an SPM as the image of an order projection of a poset which in an appropriate sense is easier to understand.
This extends Reading's corresponding construction for Bruhat intervals \cite[Section~5]{Reading}.

Let $P$ be a finite poset with $\hat0$ and $\hat1$. Assume $M$ is an SPM on $P$, and define $\pi: [\hat0,M(\hat1)] \times \2 \to P$ by
\[
(p,\ga) \mapsto \begin{cases}M(p) & \h{if $\ga=\be$ and $p \cov M(p)$,} \\ p & \h{otherwise.}\end{cases}
\]
It is readily checked that the fibres of $\pi$ are as follows:
\begin{equation} \label{eq:fibres}
\pi^{-1}(p)=\begin{cases}\{(M(p),\be)\} & \h{if $p \not\leq M(\hat1)$,} \\
\{(p,\al)\} & \h{if $p<M(p)$,} \\
\{(p,\al),(p,\be)\} & \h{if $p=M(p)$,} \\
\{(p,\al),(M(p),\be),(p,\be)\} & \h{if $M(p)<p \leq M(\hat1)$.}\end{cases}
\end{equation}

\begin{Lemma}
The map $\pi$ is an order projection. In particular, $P$ is isomorphic to the fibre poset $\fib_\pi$.
\end{Lemma}

\begin{proof}
For brevity, define $Q=[\hat0,M(\hat1)] \times \2$. First we show that $\pi:Q \to P$ is order-preserving.
Suppose $(p',\ga') \leq (p,\ga)$ in $Q$. The only non-obvious case to consider is when $\pi((p',\ga'))=M(p')$.
Then, if $\pi((p,\ga))=M(p)$, $M(p') \leq M(p)$ follows from the lifting property since $p<M(p)$ in this case.
If, instead, $\pi((p,\ga))=p$ we have $M(p)\leq p$ because $\ga=\be$. Hence, lifting yields $M(p') \leq p$, as desired. Thus $\pi$ is order-preserving.

Now assume $p' \leq p$ in $P$. We have to produce $q' \in \pi^{-1}(p')$ and $q \in \pi^{-1}(p)$ such that $q' \leq q$ in $Q$.
\begin{itemize}
\item If $p \leq M(\hat1)$, we may use $q'=(p',\al)$ and $q=(p,\al)$.
\item If $p \not\leq M(\hat1)$ and $M(p') \geq p'$, use $q'=(p',\al)$ and $q=(M(p),\be)$; lifting first implies $M(p)<p$ and then $p' \leq M(p)$.
\item Finally, if $p \not\leq M(\hat1)$ and $M(p')<p'$, we may take $q'=(M(p'),\be)$ and $q=(M(p),\be)$; here lifting first yields $M(p)<p$ and then $M(p') \leq M(p)$.
\end{itemize}
Thus $\pi$ is an order projection. By Lemma~\ref{le:fibre}, $P$ and $\fib_\pi$ are isomorphic.
\end{proof}

The previous lemma describes a poset with an SPM as a fibre poset.
Next, we show that the fibre poset can be constructed from the domain of the order projection using modifications that change the topology in a controlled manner.
This is analogous to Reading's \cite[\Th~5.5]{Reading}.

\begin{Thm} \label{th:convert}
Let $P$ be a finite poset with $\hat0$ and $\hat1$.
If $M$ is an SPM on $P$, then $P$ can be obtained from $[\hat0,M(\hat1)] \times \2$ by a sequence of clean zippings and removals.
\end{Thm}

\begin{proof}
Again, let $Q=[\hat0,M(\hat1)] \times \2$. Suppose $F_1,\p,F_t$ is a linear extension of the fibre poset $\fib_\pi$ of the order projection $\pi: Q \to P$. This means that
\begin{equation} \label{eq:classes}
F_k \ni x \leq y \in F_l \To k \leq l.
\end{equation}

Consider the sequence of posets $Q=P_0,P_1,\p,P_t=\fib_\pi \cong P$, where $P_i$ is obtained from $P_{i-1}$ by identifying the elements of $F_i$.
More precisely, as sets,
\[
P_i=\Bigg(Q-\bigcup_{j=1}^i{F_j}\Bigg) \cup \{F_1,\p,F_i\},
\]
and the order on $P_i$ is given by $a \leq_{P_i}b$ if and only if (i) $a,b \in Q$ and $a \leq_Q b$, (ii) $a=F_k$, $b \in Q$, and $x \leq_Q b$ for some $x \in F_k$,
or (iii) $a=F_k$, $b=F_l$, and $x \leq_Q y$ for some $x \in F_k$ and $y\in F_l$.

Clearly, $P_i \cong P_{i+1}$ if $|F_{i+1}|=1$.
It suffices to prove that $P_{i+1}$ is obtained from $P_i$ by a removal if $|F_{i+1}|=2$ and by a clean zipping if $|F_{i+1}|=3$.

Suppose first $|F_{i+1}|=2$, so that $F_{i+1}=\{(p,\al),(p,\be)\}$ for some $p=M(p)$. We must show that $(p,\be)$ only covers $(p,\al)$ in $P_i$.
By \eqref{eq:classes}, all other elements below $(p,\be)$ in $P_i$ are of the form $F_k$, $k \leq i$.
Suppose $(p,\be)$ covers $F_k$. Then there exists $(p',\be) \in F_k$ such that $p$ covers $p'$ in $P$.
Since $M$ is an SPM, $M(p') \leq p'$, which, by \eqref{eq:fibres}, implies $(p',\al) \in F_k$. Hence $(p,\al)>_{P_i} F_k$, which is the desired contradiction.

The conditions on removable elements that are left to check involve only the structure of the order filter generated by $(p,\al)$.
By \eqref{eq:classes}, this order filter in $P_i$ is equal to the same order filter in $Q$.
In $Q$, however, the conditions are obvious (as the coatom $c$, take $(M(\hat1),\al)$).

Second, assume $|F_{i+1}|=3$ with $F_{i+1}=\{(p,\al),(M(p),\be),(p,\be)\}$; in particular, this means that $M(p)<p \leq M(\hat1)$.
We have to show that $((p,\al),(M(p),\be),(p,\be))$ is a clean zipper in $P_i$.
That $(p,\be)$ only covers $(p,\al)$ and $(M(p),\be)$ in $P_i$ is shown in the same way as when $|F_{i+1}|=2$.
Next, let us verify that $(p,\al)$ and $(M(p),\be)$ are above the same elements. By \eqref{eq:classes}, only fibres $F_k$, $k \leq i$, need to be considered.
So, suppose $F_k <_{P_i} (p,\al)$. That is, there exists $(p',\al) \in F_k$ with $p'<p$.
\begin{itemize}
  \item If $M(p')<p'$, $(M(p'),\be) \in F_k$. The lifting property asserts that $M(p')<M(p)$. Therefore, $F_k <_{P_i} (M(p),\be)$.
  \item If $M(p') \geq p'$, lifting implies $p' \leq M(p)$. Hence, $(p',\al) <_Q (M(p),\be)$ so that, again, $F_k <_{P_i} (M(p),\be)$.
\end{itemize}
Now, suppose instead $F_k <_{P_i} (M(p),\be)$.
\begin{itemize}
  \item If $(p',\al) \in F_k$ for some $p' \leq M(p)$, then $(p',\al) <_Q (p,\al)$ and $F_k <_{P_i} (p,\al)$.
  \item Otherwise, \eqref{eq:fibres} shows that $\{(M(p'),\be),(p',\al)\} \seq F_k$ holds for some $M(p') \leq p'$ and $M(p')<M(p)$.
Then the lifting property yields $p'<p$. Thus, $(p',\al) <_Q (p,\al)$ and $F_k <_{P_i} (p,\al)$.
\end{itemize}

The conditions on clean zippers that remain to be verified involve only the structure of the order filter generated by $(p,\al)$ and $(M(p),\be)$.
As before, the conditions hold in $Q$, hence in $P_i$ by \eqref{eq:classes}.
\end{proof}

\section{Main results} \label{se:main}

Combining the material of the previous two sections, we obtain strong topological statements about posets with special partial matchings.
These assertions, which are recorded in this section, form our main results.

\begin{Thm} \label{th:main}
Let $P$ be a finite poset with $\hat0$ and $\hat1$, and suppose $M$ is an SPM on $P$.
If $\De(\hat0,M(\hat1))$ is a PL $d$-ball, then $\De(\ol{P})$ is a PL $(d+1)$-ball.
If $\De(\hat0,M(\hat1))$ is a PL $d$-sphere, then $\De(\ol{P})$ is a PL $(d+1)$-ball or a PL $(d+1)$-sphere;
the latter holds if and only if $M$ is actually a special matching.
\end{Thm}

\begin{proof}
It follows from Lemma~\ref{le:extension} that \smash{$\De\Big(\ol{[\hat0,M(\hat1)] \times \2}\Big)$} is a PL $(d+1)$-ball (sphere) if
$\De(\hat0,M(\hat1))$ is a PL $d$-ball (sphere).
According to \Th~\ref{th:convert}, a sequence of clean zippings and removals converts $[\hat0,M(\hat1)] \times \2$ into $P$.
Moreover, removals are used precisely when $M$ has fixed points; this follows from the proof of \Th~\ref{th:convert}.

By \Ths~\ref{4.7}, \ref{MT1}, and \ref{MT2}, $\De(\ol{P})$ is a PL $(d+1)$-ball or sphere, the latter occurring precisely when
$\De(\hat0,M(\hat1))$ is a sphere and $M$ has no fixed points, i.e., is a special matching.
\end{proof}

Let us now formally define the notions of zircons and pircons, which were discussed in the introduction.
Given a poset $P$, recall that $P_{\leq x}=\{y \in P \mid y \leq x\}$.

\begin{Defn}\label{de:zircon}
A poset $P$ is a \emph{zircon} if, for every non-minimal element $x \in P$, the order ideal $P_{\leq x}$ is finite and admits a special matching.
\end{Defn}

Actually, Marietti~\cite{Marietti} originally defined zircons in a slightly different way.
His definition and \Def~\ref{de:zircon} are, however, equivalent; see \cite[\Pro~2.3]{Hultman_zircon}.
It is obvious how to generalise this to the SPM setting:

\begin{Defn}
A poset $P$ is a \emph{pircon} if, for every non-minimal element $x \in P$, the order ideal $P_{\leq x}$ is finite and admits an SPM.
\end{Defn}

Clearly, zircons are pircons. Recall from the introduction that all open intervals in zircons are topological spheres. This characterises zircons among pircons:

\begin{Thm} \label{th:pircon}
Suppose $P$ is a pircon and $x<y$ in $P$. Then $\dxy$ is a PL ball or a PL sphere.
Moreover, there exist $x<y$ in $P$ such that $\dxy$ is a ball if and only if $P$ is not a zircon.
\end{Thm}

\begin{proof}
First, observe that every principal order ideal $P_{\leq y}$ has a unique minimum.
Indeed, the lifting property shows that every minimal element in $P_{\leq y}$ also belongs to $P_{\leq M(y)}$, where $M$ is an SPM on $P_{\leq y}$.
The observation now follows by induction on the cardinality of a longest chain in the ideal.

Let $\hat0$ be the minimum of $P_{\leq y}$. Using similar induction, we may assume $\De(\hat0,M(y))$ is a PL ball or sphere.
By \Th~\ref{th:main}, $\De(\hat0,y)$ is a PL ball or sphere, too. The same holds for $\dxy$ since it is a link in $\De(\hat0,y)$.

For the final statement, we know that open intervals in zircons are spheres. On the other hand, if $P$ is not a zircon, some $P_{\leq y}$ admits an SPM with fixed points.
\Th~\ref{th:main} then shows that $\De(\hat0,y)$ is a ball, where again $\hat0$ is the minimum of $P_{\leq y}$.
\end{proof}

\section{Pircons in Coxeter group theory} \label{se:coxeter}

In this section, we demonstrate how \Th~\ref{th:pircon} can be applied to certain posets appearing in Coxeter group theory.
Acquaintance with the basics of this theory, as explained for example in \cite{B-B} or \cite{Humphreys}, is assumed.

\subsection{Twisted identities} \label{sse:twist}

As a first application, we shall prove \cite[\Con~6.3]{Hultman3}. The reader may consult \cite{Hultman3} for context.
Here we only describe the necessary ingredients for the statement and its proof.

Let $(W,S)$ be a Coxeter system with an involutive automorphism $\t$. Define two subsets of $W$ as follows. The set of \emph{twisted involutions} is
\[
\twist=\{w \in W \mid \t(w)=w^{-1}\},
\]
and the set of \emph{twisted identities} is
\[
\twid=\{\t(w)w^{-1} \mid w \in W\}.
\]
It is clear that $\twid \seq \twist$.

Say that $\t$ has the \emph{no odd flip}, or \emph{NOF}, \emph{property} if $s\t(s)$ has even or infinite order for every $s \in S$ with $s \neq \t(s)$.\footnote{This means that $\t$ does not flip any edges with odd labels in the Coxeter graph.}
For any $X \seq W$, let $\Br(X)$ denote the subposet of the Bruhat order on $W$ which is induced by $X$.
The identity element $e \in W$ is the minimum in $\Br(W)$, hence in $\Br(\twid)$.

The poset $\Br(\twist)$ is always graded; denote its rank function by $\rho$. Whenever $\Br(\twid)$ is graded, its rank function is the restriction of $\rho$.
Furthermore, $\Br(\twid)$ is graded if $\t$ satisfies the NOF property \cite{Hultman3}.

When $W$ is of type~$A_{2n+1}$ and $\t$ is the unique non-trivial involution, \cite[\Th~4.3]{A-H} shows that $\Br(\twid)$ is a pircon.
This is generalised substantially in the next result. The main proof ideas are, however, the same.

\begin{Thm} \label{th:nof}
If $\t$ has the NOF property, then $\Br(\twid)$ is a pircon.
\end{Thm}

\begin{proof}
Choose $w \in \twid$ and $s \in S$ such that $ws<w$ in the Bruhat order. For $x \in \Br(\twid)_{\leq w}$, put $M(x)=\t(s)xs$.
We shall prove that $M$ is an SPM on this (finite) order ideal.

Observe that
\[
M(x)=\begin{cases}\f(x) & \h{if $\f(x) \in \twid$,} \\ x & \h{otherwise,}\end{cases}
\]
where the map
\[
\f(x)=\begin{cases}xs & \h{if $M(x)=x$,} \\ M(x) & \h{otherwise}\end{cases}
\]
is a special matching on $\Br(\twist)_{\leq w}$ by \cite[\Th~4.5]{Hultman2}.
Hence, $M$ preserves $\Br(\twid)_{\leq w}$ by the lifting property applied to $\f$.

It follows from \cite{Hultman3} that for $x \in \Br(\twid)$,
\begin{equation} \label{eq:fact}
M(x)=x \To \f(x)>x.
\end{equation}
Therefore, the second property of an SPM (see \Def~\ref{de:spm}) holds, and the first and third properties are readily checked. It remains to verify the fourth.

Suppose $x \cov y$ in $\Br(\twid)_{\leq w}$ and $M(x) \neq y$. Since $\Br(\twid)$ has the induced rank function of $\Br(\twist)$, $x \cov y$ in $\Br(\twist)_{\leq w}$, too.
We have to show that $M(x)<M(y)$. Since $\f$ is a special matching, this is obvious if $M(x) \neq x$ and $M(y) \neq y$.
Apart from some trivial cases, we thus have to consider (1) $M(x)=x$ and $M(y)<y$, and (2) $M(x)>x$ and $M(y)=y$. However, we shall see that both cases are impossible.

In the former case, by \eqref{eq:fact} we have $\f(x)>x \neq \f(y)<y$, which contradicts the lifting property.
In the latter case, \eqref{eq:fact} implies $\f(y) \voc y$.
Since $\f(y) \voc \f(x)$, too, we have a contradiction because according to \cite[Lemma~4.5]{Hultman3}, under the NOF assumption, an element in $\twist-\twid$ can cover at most one twisted identity in $\Br(\twist)$.
\end{proof}

\begin{Remark}
In general, \Th~\ref{th:nof} is false without the NOF assumption.
For example, suppose $W$ is of type~$A_4$ with generating set $S=\{s_1,s_2,s_3,s_4\}$ such that $s_1s_2$, $s_2s_3$, and $s_3s_4$ have order $3$, and all other generator pairs commute.
Let $\t$ be the unique non-trivial involution of $(W,S)$, mapping $s_i$ to $s_{5-i}$. Define $w=s_2s_1s_3s_2s_4s_3$.
One readily checks that $\Br(\twist)_{\leq w}$ is isomorphic to the rank $3$ boolean lattice, and that $\Br(\twid)_{\leq w}$ is obtained from $\Br(\twist)_{\leq w}$ by removing the rank $2$ element $s_2s_3s_2$.
The resulting poset does not admit an SPM, hence $\Br(\twid)$ cannot be a pircon.
\end{Remark}

In light of \Th~\ref{th:pircon}, \Th~\ref{th:nof} immediately implies the following result, which is the previously mentioned conjecture.

\begin{Cor}[\mbox{\cite[\Con~6.3]{Hultman3}}] \label{co:main}
Suppose $\t$ has the NOF property and let $I$ be an open interval in $\Br(\twid)$. Then $\De(I)$ is a PL ball or a PL sphere.
\end{Cor}

\begin{Remarks}
${}$

1. Can, Cherniavsky, and Twelbeck~\cite{C-C-T} established \Co~\ref{co:main} for $W$ of type~$A_{2n+1}$ using shellability methods.

2. It follows from \cite[\Th~4.12]{Hultman3} that $\De(I)$ is a sphere precisely when $I$ is \emph{full}, meaning that it coincides with an interval in $\Br(\twist)$, i.e., $I=\{x \in \twid \mid u<x<w\}=\{x \in \twist \mid u<x<w\}$ for some $u,w \in \twid$.

3. The remark after \Th~\ref{th:nof} shows that $\Br(\twid)$ is not a pircon if $W$ is of type~$A_{2m}$, $m \geq 2$, with $\t \neq \id$.
It is, however, an open question whether the open intervals are PL balls or spheres. This is not true for arbitrary $W$ and $\t$.
For example, as shown in \cite[\Exa~4.7]{Hultman3}, if $W$ is of type~$\widetilde{A}_2$ with $\t \neq \id$, there are intervals in $\Br(\twid)$ which are not even graded.
\end{Remarks}

\subsection{Quasiparabolic \texorpdfstring{$W$}{}-sets} \label{sse:quasi}

Our second application concerns quasiparabolic $W$-sets as introduced by Rains and Vazirani~\cite{R-V} as a context to which many nice properties of parabolic quotients extend.
Let us recall some crucial definitions and results from \cite{R-V}. The reader should consult the original source for much more background and motivation.

Again $(W,S)$ denotes a Coxeter system.
Say that $X$ is a \emph{scaled $W$-set} if $X$ is a (left) $W$-set equipped with a function $\hgt: X \to \Z$ such that $|\hgt (sx)-\hgt(x)| \leq 1$ for all $x \in X$ and all $s \in S$. An element $x \in X$ is called \emph{$W$-minimal} if $\hgt(x) \leq \hgt(sx)$ for all $s \in S$.
Say that $X$ is \emph{bounded from below} if the function $\hgt$ is bounded from below.

Let $T=\{wsw^{-1} \mid w \in W, \: s \in S\}$ denote the set of reflections.

\begin{Defn}[\mbox{\cite[\Def~2.3]{R-V}}] \label{de:QB}
A scaled $W$-set $X$ is called \emph{quasiparabolic} if it satisfies the following two properties.
\begin{enumerate}
  \item For all $t \in T$ and $x \in X$, if $\hgt(tx)=\hgt(x)$, then $tx=x$.
  \item For all $t \in T$, $x \in X$, and $s \in S$, if $\hgt(tx)>\hgt(x)$ and $\hgt(stx)<\hgt(sx)$, then $tx=sx$.
\end{enumerate}
\end{Defn}

\begin{Lemma}[\mbox{\cite[\Co~2.10]{R-V}}] \label{le:minmax}
Each orbit of a quasiparabolic $W$-set contains at most one $W$-minimal element.
\end{Lemma}

Suppose now that $X$ is quasiparabolic with a $W$-minimal element $x_0$. Assume, \WLOGtwo, that $\hgt(x_0)=0$.
If $y \in X$ with $\hgt(y)=k$, then $s_1 \cd s_kx_0$ is a \emph{reduced expression} for $y$ if $y=s_1 \cd s_kx_0$ for some $s_i \in S$.
All elements in the orbit of $x_0$ have reduced expressions \cite{R-V}.
Define the \emph{Bruhat order} $\leq$ on $X$ as follows.

\begin{Defn}[\mbox{\cite[\Th~5.15]{R-V}}] \label{de:Bruhat}
Let $y=s_1 \cd s_k x_0$ be a reduced expression. Then
\[
\h{$x \leq y \iff x=s_{i_1} \cd s_{i_j}x_0$ for some $1 \leq i_1<\cd<i_j \leq k$.}
\]
\end{Defn}

In particular, elements in different $W$-orbits are incomparable.
Although not obvious from \Def~\ref{de:Bruhat}, the Bruhat order is a partial order on $X$, which we denote by $\Br(X)$; it is graded with rank function $\hgt$ \cite{R-V}.
In particular, $W$-minimal elements are minimal in the Bruhat order.

Again there is a ``lifting property'':

\begin{Lemma}[\mbox{\cite[Lemma~5.7]{R-V}}] \label{le:lift}
Suppose $x,y \in X$ and $s \in S$. If $x \leq y$ and $sx \not\leq sy$, then $sx \leq y$ and $x \leq sy$.
\end{Lemma}

\begin{Thm} \label{th:quasipircon}
If $X$ is a quasiparabolic $W$-set bounded from below, then $\Br(X)$ is a pircon.
In particular, the order complex of every open interval in $\Br(X)$ is a PL ball or a PL sphere.
\end{Thm}

\begin{proof}
Suppose $z\in X$ is a non-minimal element. Since $X$ is bounded from below, there is a minimal element $x_0<z$.
By Lemma~\ref{le:minmax}, $x_0$ is in fact unique since elements in different $W$-orbits are incomparable. Hence $\Br(X)_{\leq z}=[x_0,z]$.
By \Def~\ref{de:Bruhat}, $[x_0,z]$ is finite.
Choose a reduced expression $s_1 \cd s_kx_0$ for $z$. For $x \in [x_0,z]$, let $M(x)=s_1x$. We shall prove that $M$ is an SPM on $[x_0,z]$.
\begin{itemize}
  \item For all $x \leq z$, $s_1s_1x=x$. Thus $M^2=\id$.
  \item Since $\hgt(s_1z)=\hgt(s_2 \cd s_kx_0)=k-1$, $M(z) \cov z$. Lemma~\ref{le:lift} thus shows that $M(x) \leq z$ for all $x \leq z$.
  \item For all $x \leq z$, $s_1x$ and $x$ are comparable by \cite[Remark~5.2]{R-V}, and $|\hgt(s_1x)-\hgt(x)| \leq 1$. Hence, $M(x) \cov x$, $M(x)=x$, or $x \cov M(x)$.
  \item Suppose $x \cov y \leq z$ and $M(x) \neq y$. Then $s_1x \neq y$, $x \neq s_1y$, and $s_1x \neq s_1y$.
By Lemma~\ref{le:lift}, we either have $s_1x<s_1y$, or else $s_1x<y$ and $x<s_1y$. In the latter case, $s_1x \not> x$, so $s_1x \leq x<s_1y$.
Hence, in either case, $M(x)<M(y)$. \qedhere
\end{itemize}
\end{proof}

The topological conclusion of \Th~\ref{th:quasipircon} is implied by \cite[\Th~6.4]{R-V}, which claims CL-shellability of the intervals.
Unfortunately, the proof of that result has turned out to be flawed; see the discussion in \cite{C-C-T}.

A familiar example of a quasiparabolic $W$-set is the parabolic quotient $W^J$, $J \seq S$, which consists of the minimal length representatives of the left cosets of the parabolic subgroup $W_J$ in $W$.
In this setting, the topological conclusion of \Th~\ref{th:quasipircon} was established by Bj\"orner and Wachs~\cite{B-W_coxeter} using shellability techniques.

Other examples include several instances of $\twid$ (with $W$ acting by twisted conjugation, i.e., the action of $w$ on $x$ is given by $wx\t(w^{-1})$), including the odd rank type~$A$ case \cite{R-V}.
In fact, it seems possible that $\twid$ is always a quasiparabolic $W$-set with this action whenever $\t$ has the NOF property; if so, \Th~\ref{th:nof} would be a special case of \Th~\ref{th:quasipircon}.
We neither know of a proof nor of a counterexample.

\section{Open questions} \label{se:questions}

We conclude the paper with a couple of questions that suggest themselves naturally.

Clearly, all zircons and pircons have rank functions.\footnote{For zircons, the existence of a rank function is part of Marietti's~\cite{Marietti} original definition;
see the discussion after \Def~\ref{de:zircon}.}
Indeed, the rank of an element $x$ equals the dimension of the ball or sphere $\De(\hat0,x)$ plus two, where $\hat0$ is the unique minimal element below $x$;
the uniqueness was shown in the proof of \Th~\ref{th:pircon}.

Let $Z$ be a zircon with rank function $\rk$. For a non-minimal element $z \in Z$, let $M_z$ denote a fixed special matching on $Z_{\leq z}$.
Given an induced subposet $P \seq Z$ and $p \in P$, define
\[
M_p'(x)=\begin{cases}M_p(x) & \h{if $M_p(x) \in P$,} \\ x & \h{otherwise.}\end{cases}
\]
Suppose $M_p'$ is an SPM on $P_{\leq p}$ for every non-minimal element $p \in P$.
If, moreover, the restriction of $\rk$ to $P$ is a rank function of $P$, call $P$ an \emph{induced pircon} of $Z$.

It follows from the proof of \Th~\ref{th:nof} that every pircon of the form $\Br(\twid)$ is an induced pircon of the corresponding zircon $\Br(\twist)$.
Similarly, $\Br(W^J)$ is an induced pircon of $\Br(W)$ for any $J \seq S$.

\begin{Ques}
Is every pircon an induced pircon of some zircon?
\end{Ques}

A common way to establish topological consequences such as those of \Th~\ref{th:pircon} is to prove shellability.
Beginning with Bj\"orner~\cite{Bjorner}, there are several variations of lexicographic shellability; see, e.g., Wachs' survey~\cite{Wachs}.
Under this umbrella are gathered several similarly flavoured combinatorial methods that can be used to establish shellability of order complexes by means of certain labellings of the posets.

Concerning zircons, the following question is known to have an affirmative answer for $\Br(W)$ in arbitrary type \cite{B-W_coxeter},
as well as for $\Br(\twist)$ in types~$A$, $B$, and $D$ \cite{Incitti_A,Incitti_B,Incitti_D}.
For other pircons, it has been established for $\Br(W^J)$ \cite{B-W_coxeter} and for $\Br(\twid)$ in type~$A$ of odd rank \cite{C-C-T}.

\begin{Ques}
Is every interval in every pircon lexicographically shellable?
\end{Ques}

In case both the previous questions turn out to have affirmative answers, one may speculate that even more could be true.
The aforementioned result from \cite{C-C-T} can be interpreted in the following way.
For $W$ of type~$A_n$, Incitti~\cite{Incitti_A} established lexicographic shellability of $\Br(\twist)$ by producing an EL-labelling of this poset.
When $n$ is odd and $\t \neq \id$, Can, Cherniavsky, and Twelbeck proved that the restriction of this labelling to the induced pircon $\Br(\twid)$ is an EL-labelling, too.

\begin{Ques}
Is it true that every induced pircon has an EL-labelling which is induced from an EL-labelling of the corresponding zircon?
\end{Ques}

\section*{Acknowledgements}

N. Abdallah was funded by a stipend from the Wenner-Gren Foundations.

\bibliographystyle{amsplain}

\bibliography{Referenser_bara_initialer}

\end{document}